\documentclass[11pt,reqno]{amsart}
\usepackage[a4paper,left=24mm,right=24mm,top=32mm,bottom=30mm]{geometry} 

\usepackage[utf8]{inputenc}

\usepackage{amsmath,amsfonts,amsthm}
\usepackage{enumerate}

\usepackage{color}

\newcommand\nc{\newcommand}

\nc{\N}{\mathbb{N}}
\nc{\abs}[1]{\left\lvert #1 \right\rvert}
\nc{\dt}{\mathrm{d}t}
\nc{\weak}{\rightharpoonup}
\nc{\eps}{\epsilon}
\nc{\Mino}{\mathbb{M}}
\nc{\weakstar}{\stackrel{*}{\rightharpoonup}}

\renewcommand{\div}{\operatorname{div}}

\nc{\od}[2]{\frac{d#1}{d#2}}
\nc{\be}{\begin{equation}}
\nc{\ee}{\end{equation}}
\nc{\bd}{\begin{displaymath}}
\nc{\ed}{\end{displaymath}}
\nc{\bq}{\begin{eqnarray}}
\nc{\eq}{\end{eqnarray}}
\nc{\p}{\partial}
\nc{\ra}{\rightarrow}
\nc{\R}{\mathbb{R}}
\nc{\dd}{\mathrm{d}}
\nc{\dx}{~\mathrm{d}x}
\nc{\curl}{\operatorname{curl}}
\nc{\dv}{\operatorname{div}}
\nc{\meas}{\operatorname{meas}}
\nc{\supp}{\operatorname{supp}}
\nc{\medotimes}{\resizebox{11pt}{!}{$\otimes$}}
\nc{\T}{\mathrm{T}}
\nc{\wk}{\rightharpoonup}
\nc{\ep}{\varepsilon}
\nc{\hep}{\hat{\ep}} 
\nc{\Id}{\mathrm{Id}}
\nc{\norm}[1]{\left\|#1\right\|}
\nc{\pl}{\mathrm{pl}}
\nc{\el}{\mathrm{el}}
\nc{\Curl}{\operatorname{Curl}}
\nc{\B}{\mathcal{B}}
\nc{\M}{\mathcal{M}}
\nc{\G}{\mathcal{G}}

\nc{\commentK}[1]{{\color{red}[KA comment:~#1]}}
\nc{\commentM}[1]{{\color{blue}[MK comment:~#1]}}
\nc{\commentP}[1]{{\color{green}[PWD comment:~#1]}}

\theoremstyle{plain}
\newtheorem{theorem}{Theorem}
\newtheorem{lemma}[theorem]{Lemma}
\newtheorem{proposition}[theorem]{Proposition}

\theoremstyle{definition}

\theoremstyle{remark}
\newtheorem{remark}{Remark}
\newtheorem{example}{Example}

\newcommand{\EOR}{\color{black}}

\begin{document}

\author[Keith~Anguige]{Keith Anguige, Patrick Dondl}
\address{Abteilung für Angewandte Mathematik,
	Albert-Ludwigs-Universität Freiburg,
	Hermann-Herder-Str. 10,
	79104 Freiburg i. Br., Germany}
\email{keith.anguige@mathematik.uni-freiburg.de} \email{patrick.dondl@mathematik.uni-freiburg.de}

\author{Martin Kru\v{z}\'{\i}k}
\address{Institute of Information Theory and Automation,
Czech Academy of Sciences, Pod vod\'{a}renskou
v\v{e}\v{z}\'{\i}~4, CZ-182~08~Praha~8, Czech Republic (corresponding
address) \& Faculty of Civil Engineering, Czech Technical
University, Th\'{a}kurova 7, CZ-166~ 29~Praha~6, Czech Republic}
\email{kruzik@utia.cas.cz}

\keywords{Existence of minimizers, single-crystal plasticity, cross-hardening, geometrically necessary dislocations, strain-gradient}

\subjclass[2010]{49J10, 49J45, 74G25, 74G65, 74N15}

\title{On the existence of minimisers for strain-gradient single-crystal plasticity}

\begin{abstract}We prove the existence of minimisers for a family of models related to the single-slip-to-single-plane relaxation of single-crystal, strain-gradient elastoplasticity with $L^p$-hardening penalty. In these relaxed models, where only one slip-plane normal can be activated at each material point, the main challenge is to show that the energy of geometrically necessary dislocations is lower-semicontinuous along bounded-energy sequences which satisfy the single-plane condition, meaning precisely that this side condition should be preserved in the weak $L^p$-limit. This is done with the aid of an `exclusion' lemma of Conti \& Ortiz, 
 which essentially allows one to put a lower bound on the dislocation energy at interfaces of (single-plane) slip patches, thus precluding fine phase-mixing in the limit. Furthermore, using div-curl techniques in the spirit of Mielke \& M\"uller, we are able to show that the usual multiplicative decomposition of the deformation gradient into plastic and elastic parts interacts with weak convergence and the single-plane constraint in such a way as to guarantee lower-semicontinuity of the (polyconvex) elastic energy, and hence the total elasto-plastic energy, given sufficient ($p>2$) hardening, thus delivering the desired result.
\end{abstract}
 
\maketitle 

\section{Introduction and main results}

Plastic deformation in crystals has long been known to be mediated by the motion of crystal imperfections, or {\em dislocations}, through the material. Each such dislocation travels predominantly on a given crystallographic plane, a so-called {\em slip-plane}, in the direction of a fixed {\em Burgers' vector}, where the Burgers' vectors are also determined by the crytallographic structure. As a consequence of this, the plastic strain can be thought of as the product of a number of simple-shear deformations, each with a given crystallographically determined shear normal and shear direction.
  
In numerous experiments, lamination-type microstructures with alternating layers of slip-system activity have been observed~\cite{Rasmussen_80a, Jin_84a, Bay_92a,dm}, and this effect, which will be a central consideration in what follows, is widely believed to be a consequence of {\em cross-hardening}~\cite{ortiz,dm}, the phenomenon whereby activity in one slip system at a given point suppresses activity in all other slip systems at that point. A microscopic explanation for the effect is that it arises from the formation of energetically favorable (and \emph{sessile}) dislocation products when two dislocations from different slip-planes meet -- these {\em Lomer-Cottrell locks} have been observed in experiments~\cite{Lee_13}\EOR, and also studied in detail in atomistic simulations~\cite{Rodney:1999vs}. The idea is that, in order to continue the plastic deformation with activity in more than one slip-plane, either these locks have to be broken or new dislocation loops must be formed, thus necessitating an increased energy input for plastic deformations involving activity in multiple slip-systems.

In their seminal article \cite{ortiz}, Ortiz and Repetto  proposed a method for reproducing experimentally observed sub-grain pattern formation in plastically deforming crystals. The basic idea, which we will follow in this paper, is to model plastic evolution by an incremental time-stepping procedure, where in each step the sum of an elastic energy and a stored plastic energy is minimised. They then propose a non-convex stored energy in order to account for cross-hardening, which leads to the formation of microstructure.

As in \cite{ang,ang2}, we will introduce the aforementioned non-convexity generated by cross-hardening into this framework in the simplest possible way, namely by enforcing a hard single-slip condition, such that at each material point any finite-energy plastic deformation has to occur in single slip. We then relax this single-slip condition to a single-plane condition (namely that only one slip-plane can be active at each point) with the aid of a laminated microstructure -- similar side-conditions have been considered by several other authors; see, for example~\cite{dolz,dolz2}. In~\cite{dm} it was shown that the predicted laminate microstructure arising from this assumption of infinite cross hardening does indeed match experimental results, while evolutionary models of such laminate structures have been analysed in~\cite{Hackl1, Hackl2}.

Here, we continue our previous investigations \cite{ang,ang2}, which focused on optimal energy scalings and relaxation of the single-slip condition to a (still non-convex) single-plane condition, by looking at the existence question for a class of incremental minimisation problems which arise in the way described above. Our main goal is to extend the existence result of Mielke and M\"uller~\cite{MM} for finite, multiplicative strain-gradient elasto-plasticity to the single-crystal, single-plane case. It turns out that the single-plane restriction on the plastic slip allows one nicely to control the inverse of the plastic deformation, and, with the aid of a sufficiently strong div-curl lemma, the weak continuity of the minors of the elastic deformation -- as a consequence, this allows for more general (and realistic) material parameters than in \cite{MM}. On the other hand, proving that the non-convex single-plane constraint is preserved along energy-minimising sequences, which is the key lower-semicontinuity property, requires considerable additional analysis. In particular, we have to adapt an exclusion Lemma of \cite{co}, originally designed for a 2-d problem, to our 3-d setting in such a way as to exclude any further phase mixing in our ostensibly relaxed single-plane models.

The family of single-plane incremental energies treated here is related to the {\em single-plane} relaxation of the {\em  single-slip} model discussed in \cite{ang,ang2}, and contains a strain-gradient penalisation of geometrically necessary dislocations, along with an $L^p$-stored-energy term. We examine this relationship later on, and, while no explicit formula for the relaxation seems to be available in general, we are nevertheless able to write down matching upper and lower bounds for the relaxed single-plane energy functional which ensure the applicability of our main result, namely the following. 

\begin{theorem}\label{thm:main}
Let $\Omega \subset \R^3$ be a bounded, Lipschitz domain, and let $p>2$. Let $m_j$, for $j=1,\ldots N$, be a family of slip normals, and $s_j\in m_j^{\perp}$ the corresponding slip vectors. Further suppose that $W_{\mathrm{el}}$ is a polyconvex, frame-indifferent elastic-energy density satisfying the growth condition 
\[
W_{\mathrm{el}}(A) \ge -c_1 + c_2\abs{A}^q\quad \textrm{with}\quad q>p/(p-2),
\]
and that the single-plane plastic energy, $E_\pl$, is w.l.s.c. in $(L^p(\Omega))^N$, and satisfies 
\[
E_\pl\left(\{s_j\}_{j=1}^N\right)\geq C\left(\sum_{j=1}^N \|{s_j}\|_{L^p(\Omega)}^p + \sum_{j=1}^N\int_{\Omega}\abs{\nabla_{m_j^\perp} s_j}\right),
\]
for some constant $C>0$.

Then, writing $E_{\mathrm{el}}\left(y, \{s_j\}_{j=1}^N\right):= \int_{\Omega} W_{\mathrm{el}}\left(\nabla y\,F_\pl^{-1}\right)dx$, where the plastic-deformation tensor is given by $F_\pl = \Id + \sum_{j=1}^N s_j \otimes m_j$ and $y(x)$ is the deformation, the single-plane elastoplastic energy
\[
E\left(y, \{s_j\}_{j=1}^N\right) =  \left\{\begin{array}{ccl}
 E_{\mathrm{el}}\left(y, \{s_j\}_{j=1}^N\right) + E_\pl\left(\{s_j\}_{j=1}^N\right) & : & \abs{s_j}\abs{s_i} = 0~\textrm{a.e., for}~i\neq j\\
 & & \\
 +\infty & : & \mathrm{otherwise}
\end{array} \right.,
\]
admits a minimiser $\left(y^{\ast},\{s^{\ast}_j\}_{j=1}^N\right)$ in the class $(W^{1,r}(\Omega))^3\times (L^p(\Omega))^N$, where $\frac{1}{r}=\frac{1}{p}+\frac{1}{q}$.
\end{theorem}

\begin{remark}
We recall that $W_{\mathrm{el}}$ is said to be polyconvex \cite{Ball_1977} if  it can be written as a convex function of the deformation gradient, its cofactor matrix, and its determinant, i.e., $W_{\mathrm{el}}(F)=h(F,{\rm Cof}\, F, {\rm det}\, F)$ for all $F$ and some convex function $h$. Moreover, frame-indifference means that $W_{\mathrm{el}}(F) =W_{\mathrm{el}}(RF)$ for all $F\in\R^{3\times 3}$ and every proper rotation $R$.
Strictly speaking, we ought to impose additional non-interpenetration conditions of the form $W_{\mathrm{el}}(F)\to+\infty$ as $\det F\to 0_+$, and $W_{\mathrm{el}}(F)=+\infty$ if $\det F\le 0$. However, these conditions will play no particular role in our analysis, and we therefore prefer to overlook them here.

\end{remark}
\begin{remark}
The non-convex side condition on the plastic slip states that, at almost every point, at most one slip-plane normal can be active, while the lower-bound on $E_\pl$ ensures that $E$ is sufficiently coercive to ensure the weak-closedness of the single-plane condition along minimising sequences.

\end{remark}

\begin{remark}
Since $s_j\in m_j^{\perp}$, we have $F_\pl^{-1}(x) = \Id - s_{j(x)}\otimes m_{j(x)}$ almost everywhere for finite-energy test functions, which will simplify the analysis of minors of the elastic deformation in the sequel.
\end{remark}

\begin{remark}
The quantities $\abs{\nabla_{m_j^\perp} s_j}$ appearing in the lower bound on $E_{\mathrm{pl}}$ are to be interpreted as measures.
\end{remark}

We also prove the following analogue of the above result for linearised elasticity, with a weaker requirement on the hardening.

\begin{theorem}\label{thm:main_lin}
Let $\Omega \subset \R^3$ be a bounded, Lipschitz domain, and let $p>1$. Let $m_j$, for $j=1,\ldots N$, be a family of slip normals, and $s_j\in m_j^{\perp}$ the corresponding slip vectors. Further suppose that the single-plane plastic energy, $E_\pl$, is w.l.s.c. in $(L^p(\Omega))^N$, and satisfies 
\[
E_\pl\left(\{s_j\}_{j=1}^N\right)\geq C\left(\sum_{j=1}^N \|{s_j}\|_{L^p(\Omega)}^p + \sum_{j=1}^N\int_{\Omega}\abs{\nabla_{m_j^\perp} s_j}\right),
\]
for some constant $C>0$.

Then, assuming a linearised elastic energy of the form
\[
E_{\mathrm{el}}\left(u, \{s_j\}_{j=1}^N\right):= \int_{\Omega} \left|(\nabla u - \beta)_{\mathrm{sym}}\right|^2 dx,
\]
where the plastic-distortion tensor is given by $\beta = \sum_{j=1}^N s_j \otimes m_j$ and $u$ is the displacement, the single-plane elastoplastic energy
\[
E_{\mathrm{lin}}\left(u, \{s_j\}_{j=1}^N\right) =  \left\{\begin{array}{ccl}
 E_{\mathrm{el}}\left(u, \{s_j\}_{j=1}^N\right) + E_\pl\left(\{s_j\}_{j=1}^N\right) & : & \abs{s_j}\abs{s_i} = 0~\textrm{a.e., for}~i\neq j\\
 & & \\
 +\infty & : & \mathrm{otherwise}
\end{array} \right.,
\]
admits a minimiser $\left(u^{\ast},\{s^{\ast}_j\}_{j=1}^N\right)$ in the class $(W^{1,p}(\Omega))^3\times (L^p(\Omega))^N$.
\end{theorem}

\begin{remark}
Unfortunately, one cannot expect an analogous existence result for $p=1$, corresponding to rate-independent dissipation, even in the case of linearised elasticity, due to the possibility of slip concentration (formation of singular measures) along minimising sequences, and the difficulty of reconciling this with an intuitive interpretation of the single-plane side condition for measures -- see Example \ref{ex:2} at the end of the paper.
\end{remark}

The force of our existence results is that one should always seek to relax a non-convex single-slip condition in single-crystal strain-gradient plasticity to a single-plane condition. In this way, one obtains a well-posed model which, in particular, is not plagued by the kind of fine oscillations which have often been observed in simulations of single-slip models. It is also worth emphasising that both the single-plane condition {\em and} the regularising penalisation of geometrically necessary dislocations which appear in the definition of $E$, acting in concert, are essential ingredients in proving these results.

The article is organised as follows. In Section~\ref{sec:model}, we introduce a non-convex (single-slip) model for single-crystal, strain-gradient plasticity, including a very particular penalisation of geometrically necessary dislocations which prevents the cancellation of dislocations at collinear-slip-patch interfaces. In Section~\ref{sec:relax}, we discuss how to relax the single-slip condition (the source of the non-convexity) in our model to a single-plane condition, resulting in a family of models to which Theorem \ref{thm:main} or \ref{thm:main_lin} applies. The mathematical heart of the paper is Section~\ref{sec:existence}, in which the proof of Theorems \ref{thm:main} and \ref{thm:main_lin} is to be found -- this consists of several lemmas which take care not only of the lower semi-continuity of the plastic and elastic parts of the single-plane energy, but also of the preservation of the single-plane condition along minimising sequences.
 
\section{A model for strain-gradient plasticity with cross hardening} \label{sec:model}

We now introduce, term-by-term, the elements of our continuum crystal-plasticity model. Modulo a very specific choice for the penalisation of geometrically necessary dislocations, which is essential for handling the non-convex slip conditions, the model ingredients are standard fare in the continuum-plasticity literature. Note that a longer version of this discussion appeared in the review \cite{AD_book}.
\subsection{Plastic deformation.} 
We consider an elasto-plastic body with its reference configuration $\Omega \subset \R^3$ and a sufficiently smooth deformation 
$$
y\colon \Omega \to \R^3,
$$
satisfying suitable boundary conditions -- for example, Dirichlet conditions on a subset of $\p\Omega$. We then define the deformation gradient $F = \nabla y$, noting that its row-wise curl necessarily vanishes.

Now we make the assumption that this deformation gradient can be decomposed into a product of plastic shears due to an atomistic rearrangement, $F_\pl$, which generates an intermediate configuration, followed by an elastic deformation, $F_\el$. We remark that the validity of this multiplicative decomposition is still a matter for debate -- see~\cite{Davoli_15,Reina_16} for two recent contributions to the discussion. 

With this  {\em Lee-Liu} decomposition~\cite{Lee_67} in hand, we can thus identify an elastic energy for our crystalline specimen depending only on the elastic strain, $F_\el = FF_{\pl}^{-1}$. A further assumption is that the plastic energy, in the sense of the implicit time discretisation (see above), can be written as a function of the incremental change in plastic strain. Thus, restricting ourselves to the first such time-step, we see that the incremental deformation requires an energy input of the form
$$
\int_\Omega W_\el(FF_\pl^{-1})\dx + E_{\pl}(F_\pl),
$$ 
for a suitable frame-indifferent elastic energy density $W_\el$, and a plastic energy $E_\pl$. This plastic energy contains a $p$-hardening term (or dissipation, in the case $p=1$) that penalises the $L^p$-norm of the plastic shear undergone by the crystal. In the following $E_\pl$ will furthermore be allowed to depend, amongst other things, on a (possibly singular) measure-valued functional of $F_\pl$ which takes account of geometrically necessary dislocations.

\subsection{Cross-hardening.}
Cross-hardening (or {\em latent} hardening)~\cite{WU:1991ij,KOCKS:1964te,FRANCISCO:1980bx} describes  the phenomenon whereby shear in one slip system suppresses activity in other slip systems at the same point in the crystal. This leads to a loss of convexity in the plastic energy $E_\pl$ introduced above~\cite{ortiz} -- roughly speaking, $E_\pl(F_\pl)$ will be locally minimal if $F_\pl$ is a simple shear in one of the given slip systems of the crystal, and, for general boundary constraints, one will have to dip into more than one of these local energy wells to minimise the energy globally. 

Here, we will make the simplifying assumption of {\em infinite} cross hardening, meaning that $F_\pl$ is required to be in single slip at each point. In line with this, it is thus assumed that the crystallographic structure admits a set of slip-plane normals $\M=\{m_j\}_{j=1}^N$, each with a given set of Burgers' vectors $\B_j = \{b_{ij}\}_{i=1}^{K(j)}$, and that $F_\pl$ takes the form.
\be\label{eq:fpl}
F_\pl = \Id + \sum_{j=1}^N \sum_{i=1}^{K(j)} c_{ij} m_j \otimes b_{ij},
\ee 
subject to the following
\be
\mathrm{\textbf{Single-Slip Condition (SSC)}}:~c_{ij}(x)c_{kl}(x) = 0\quad\mathrm{for~a.e.}~x\in\Omega,\quad\mathrm{if}\quad i\neq k\quad\mathrm{or}\quad j\neq l.
\ee

Note that, under this condition, the product of simple shears assumed above simplifies immediately, such that there is at most one non-zero factor at almost every point, thus justifying the representation as a sum and still ensuring that $\det F_\pl =1$ almost everywhere. Furthermore, the plastic hardening can be written in terms of the slip coefficients $c_{ij}$.

\subsection{Geometrically necessary dislocations.}

A strain-gradient penalty term is sometimes included in models of crystal plasticity, since an argument can be made that the surface where two differently sheared subdomains meet admits a density of geometrically necessary dislocations. In~\cite{Cermelli_2001} it is proposed that the correct term for this density of geometrically necessary dislocations must be 
\begin{equation}
\label{eq:curl_nonlin}
\frac{1}{\det F_\pl} (\Curl F_\pl) F_\pl^\T,
\end{equation}
(see also~\cite{MM} for a brief discussion of this matter). Here, the expression $\Curl$ denotes the row-wise curl of a matrix.

Considering the fact that our single-slip side condition yields a very specific form of $F_\pl$, it is easy to see that both the volumetric term and multiplication with $F_\pl^T$ in (\ref{eq:curl_nonlin}) are equal to identity. The GND-density therefore reduces to the simpler form
\be\label{eq:curl_lin}
\Curl F_\pl.
\ee
The expression above, however, does not account for sessile dislocations at boundaries between abutting subdomains deformed in collinear slip due to cancellation of dislocations with opposite sign. This is in disagreement with simulations by Devincre \textit{et. al.}~\cite{Devincre1, Devincre2} who observe that these cancellations are in practice not complete, such that a density of dislocations remains on the surface between the subdomains -- for a discussion of this matter in a simplified scalar model, see~\cite[Chapter 4]{co}. In order to exclude such collinear cancellations, we thus introduce a non-standard (possibly singular) measure for the dislocation density, which, defining the single-plane slips $ s_j = \sum_{i=1}^{K(j)} c_{ij}b_{ij}$, concretely takes the form
\be\label{mod_curl}
\G\left(\{s_j\}_{j=1}^{N}\right) = \sum_{j=1}^N|\nabla_{m_j^\perp} s_j|,
\ee
i.e., for the $j$-th slip normal, we take the length of the planar-gradient vector (gradient orthogonal to the respective $m_j$) of the plastic slip $s_j$, regardless of any activity in the other slip planes, and then sum over $j$.

\subsection{The model}

To summarise, the geometrically nonlinear elasto-plastic energy we use to model the phenomena above in the case of $L^p$-hardening, is taken to be

\be
\mathcal{E}^p\left(y, \{c_{ij}\}\right) =
\left\{
\begin{array}{ccl}
	\int_\Omega W_{\textrm{el}}(F_\el)\,\dx  + 
	\sigma\int_\Omega\G(\{s_j\}) + \tau\sum_{i=1}^{K(j)}\sum_{j=1}^N\int_\Omega |c_{ij}|^p\,\dx  &:& \text{if \textbf{(SSC)} holds},\\
	& & \\
	+\infty  &:&  \text{otherwise}, 
	\label{nonlinear}
\end{array}
\right.
\ee
where $W_{\textrm{el}}$ satisfies the growth, convexity and indifference conditions appearing in Theorem \ref{thm:main}, while the corresponding geometrically linear version of the model, with $\beta = \sum_{j=1}^N s_j \otimes m_j$, is
\be
\mathcal{E}^p_{\mathrm{lin}}\left(y, \{c_{ij}\}\right) =
\left\{
\begin{array}{lcl}
	\int_\Omega\left|(\nabla y - \beta)_{\mathrm{sym}}\right|^2\,\dx  + 
	\sigma\int_\Omega\G(\{s_j\}) & &\\
	& & \\
+ \tau \sum_{i=1}^{K(j)}\sum_{j=1}^N\int_\Omega |c_{ij}|^p\,\dx &:& \text{if \textbf{(SSC)} holds},\\
	& & \\
	+\infty  &:&  \text{otherwise}, 
	\label{linear}
\end{array}
\right.
\ee

As we showed in \cite{ang}, neither of these energy functionals (or indeed their analogues with, instead, a finite hardening matrix) are weakly lower-semicontinuous, due to the possibility of fine oscillations between multiple Burgers vectors with the same slip normal, and this leads to the search for a relaxed functional and a corresponding existence theorem.
\subsection{Unique decomposition of $F_{\mathrm{pl}}$ into single-plane slips}
Under certain conditions on the family of slip-normals $m_j$, there is in fact a one-to-one correspondence between $F_{\mathrm{pl}} = \mathrm{Id} + \sum_{j=1}^N s_j \otimes m_j$ and the slips $s_j$. Thus, suppose $N\in\{1,2,3,4\}$, and let $\M=\{m_j\}_{j=1}^N$ be a collection of slip-plane normals with the property that any collection of three or fewer vectors in $\M$ is linearly independent\footnote{Note that this condition is satisfied by the four slip planes of the f.c.c. crystal structure, as well as by the low-temperature slip modes (on the basal and prismatic planes) of h.c.p. crystals.}.

For each $j\in \{1,\dots, N\}$, denote by $B_j$ the space of matrices spanned by $\{s\otimes m_j : s \perp m_j \}$. The following Proposition shows that the decomposition of any traceless matrix into a linear combination of active slips, i.e. elements of $B_j$, is unique.
\begin{proposition} \label{prop:unique_decomp}
Suppose that $\beta=\sum_{j=1}^N s_j \otimes m_j$, where $s_j\in \R^3$, $s_j \perp m_j$ and $m_j\in\M$. Then, if the independence-condition on $\M$ given above holds, this slip-plane decomposition (i.e. the determination of the $s_j$) is unique.
\end{proposition}
\begin{proof}
For $N<4$, the result holds trivially. Suppose, therefore, that $N=4$, and that the claim is false. Then there exist $s_j$ $(\perp m_j)$, for  $j= 1, \dots, 4$, not all zero, such that 
\begin{equation} \label{eq:decomposition}
\sum_{j=1}^4 s_j \otimes m_j = 0.
\end{equation}
If we now take the scalar product of~\eqref{eq:decomposition} with some $m_k$, where $k\in\{1,\dots, 4\}$, on the first factor in the dyadic product, we get
\begin{align*}
0 &= \sum_{j=1}^4 (s_j\cdot m_k) m_j \\
&= \sum_{j\neq k} (s_j\cdot m_k) m_j.
\end{align*}
By the linear independence of any three $m_j$, we have $s_j \cdot m_k =0$ for all $j\neq k$. Since $k$ was arbitrary, we have $s_j= 0$ for $j = 1, \dots, 4$, a contradiction.
\end{proof}

\section{Single-plane relaxation of the single-slip condition} \label{sec:relax}

We now investigate the single-slip-to-single-plane relaxation of (\ref{nonlinear}). In other words, instead of taking the \textbf{single-slip condition}:
\be\label{strain}
F_\pl(x) = \Id+s(x)\otimes m(x),
\ee
with
\be\label{ssc}
m(x) = m_{j(x)}\in\mathcal{M} \text{ and } s(x)\in\bigcup_{i=1}^{K(j(x))} \operatorname{Span}\{b_{ij(x)}\}
\ee
almost everywhere, we simply enforce the \textbf{single-plane condition}, also referred to as the \textbf{relaxed slip condition (RSC)}, which says that (\ref{strain}) holds subject to
\be\label{rsc}
m(x) = m_{j(x)}\in\mathcal{M} \text{ and } s(x)\in\operatorname{Span}\bigcup_{i=1}^{K(j(x))} \{b_{ij(x)}\},
\ee
and look for an optimal (with respect to the energy (\ref{nonlinear}) or (\ref{linear})) single-slip approximation to a given single-plane test function. 

Suppose, then, that we have a displacement $u\in W^{1,q}$ on $\Omega\subset\mathbb{R}^3$, $q\geq 1$, satisfying a Dirichlet condition on a part of $\p\Omega$, and a relaxed plastic strain $F_\pl$ satisfying (\ref{strain}) and (\ref{rsc}), and for each slip-plane normal, $m_j$, suppose we make a fixed choice of two admissible Burgers vectors, $b_{1j}, b_{2j}$, such that
\be\label{eq:decomp}
F_\pl = \sum_{j=1}^N s_j\otimes m_j \qquad \textrm{and} \qquad s_j = \sum_{i=1}^2 c_{ij}b_{ij}.
\ee

With such a choice in hand, we aim to identify the single-plane relaxation of the single-slip energy (\ref{nonlinear}) by weakly approximating $F_\pl$ with single-slip  laminates, such that the slip direction alternates between $b_{1j}$ and $b_{2j}$ as we pass from slice to slice. In the case of $L^1$-hardening, the correct single-plane energy was already calculated explicitly in \cite{ang}, modulo a technical issue relating to the smoothness of $F_\pl$, while for hardening with $p>1$ no such explicit formula seems to be available -- in the latter case, we will instead make do with finding matching upper and lower bounds for the single-slip-to-single-plane relaxation.

\subsection{$L^1$-hardening} Suppose $p=1$. Then, for any admissible selection of Burgers vectors as above, the following functional is, morally, the relaxation of the energy $\mathcal{E}^1$:
\be\label{rel}
\mathcal{E}^1_{\mathrm{rel}}\left(y, \{s_j\}_{j=1}^N\right) = \left\{
\begin{array}{ccl}
	\int_{\Omega} W_{\textrm{el}}(F^{\mathrm{el}})\dx + \sigma\int_{\Omega}\mathcal{G}_{\mathrm{lam}}\left(\{s_j\}_{j=1}^N\right) + \tau\int_{\Omega}|F_\pl|_{\mathrm{lam}}\dx & : & \textbf{(RSC)}~\textrm{holds},\\
	& & \\
	+\infty & : & \textrm{otherwise},
\end{array}
\right.
\ee
where the {\em laminated curl} is defined by
\be
\mathcal{G}_{\mathrm{lam}}\left(\{s_j\}_{j=1}^N\right) = \sum_{j=1}^N\sum_{i=1}^2|\nabla_{m_j^{\perp}}c_{ij}|,
\ee
and the {\em laminated hardening} by
\be
|F_{\pl}|_{\textrm{lam}} = \sum_{j=1}^N\sum_{i=1}^2 |c_{ij}|.
\ee

The justification for taking (\ref{rel}) as our expression for the relaxed energy is, as we showed in \cite{ang}, that  smooth, relaxed slips $s_j$ can be approximated by  laminated single-slips $s_j^n$, such that
\be
\int_{\Omega}\mathcal{G}(\{s_j^n\}_{j=1}^N)\ra\int_{\Omega}\mathcal{G}_{\mathrm{lam}}(\{s_j\}_{j=1}^N),
\label{lam_curl}
\ee
and
\be
\sum_{j=1}^N\sum_{i=1}^2\int_{\Omega} |c_{ij}^n|\dx\ra\int_{\Omega}|F_{\pl}|_{\mathrm{lam}}\dx,
\ee
as $n\ra\infty$, and, moreover, that $W_{\textrm{el}}$ behaves continuously under lamination of $F_\pl$.

In particular, our main theorem from \cite{ang} is
\begin{theorem}\label{thm:nonlinear}
Suppose that $W_{\textrm{el}}:\mathbb{M}^{3\times 3}\mapsto [0,\infty)$ is continuous and satisfies the $q$-growth condition
\be\label{W}
-k_1 + k_2|F|^q \le W_{\textrm{el}}(F) \le K_1 + K_2|F|^q,
\ee
for some positive constants $k_1,k_2,K_1,K_2$ and $q\geq 1$. Suppose, furthermore, that we have a Lipschitz domain $\Omega\subset\mathbb{R}^3$ and a test function $(u,\{s_j\}_{j=1}^N)$ defined on $\Omega$, such that $u\in W^{1,q}$ satisfies a Dirichlet condition on a Lipschitz subset of $\p\Omega$, $F_\pl$ satisfies \textbf{(RSC)} with the $j$-th slip normal active only on $\Omega_j\subset\Omega$, and the relaxed energy (\ref{rel}) is finite. Assume  that the sets $\{\Omega_j\}_{j=1}^N$, on which $F_\pl= \textrm{Id} + s\otimes m_j$, satisfy the regularity condition $\mathcal{H}^2(\p\Omega_j\setminus\mathcal{F}\Omega_j)=0$.
	
Then, for each $\ep>0$, there exists a pair of test functions $(u_{\ep}, s_{\ep})$ satisfying the same Dirichlet condition and \textbf{(SSC)}, such that $u_{\ep}\rightharpoonup u\in W^{1,1}$, $s_{\ep}\rightharpoonup s\in L^1$ and
	\be
	\mathcal{E}^1(u_{\ep},s_{\ep}) \leq \mathcal{E}^1_{\mathrm{rel}}(u,s) + \ep. 
	\ee
\end{theorem}

\begin{remark}
The same approximation result also holds for linearised elasticity
\end{remark}

\begin{remark}
In cases where Proposition \ref{prop:unique_decomp} applies, we could equally well write $\mathcal{E}^1_{\mathrm{rel}}$ as a function of $F_{\mathrm{pl}}$ or of the $c_{1j}, c_{2j}$.
\end{remark}

Unfortunately, despite this nice characterisation of the l.s.c. envelope, an existence result for (\ref{rel}) remains elusive, since, in particular, $L^1$-control of the plastic slip is not sufficient to enforce weak-continuity of determinants of the elastic deformation in the case of nonlinear elasticity. Moreover, even for linearised elasticity, there is a problem with the relaxed side condition \textbf{(RSC)} for $p=1$, since one does not have enough coercivity to prevent slip concentration along minimising sequences (see Example \ref{ex:2}). In order to force the existence of minimisers, one can, however, add to $\mathcal{E}^1_{\textrm{rel}}$ a small, admittedly somewhat {\em ad hoc}, penalty which is bounded below by $\eps\sum_{i=1}^N\|s_i\|_p^p$ (with $p>2$ and $\epsilon> 0$) to $\mathcal{E}^1_{\textrm{rel}}$, such that the resulting energy satisfies the conditions of Theorem \ref{thm:main}. For convenience, we now state this result separately.

\begin{theorem}\label{thm:main_alt}
	Let $\Omega \subset \R^3$ be a bounded, Lipschitz domain, and suppose $p>2$, $\epsilon>0$. Let $m_j$, for $j=1,\ldots N$, be a family of slip normals, and $s_j\in m_j^{\perp}$ the corresponding slip vectors. Further suppose that $W_{\mathrm{el}}$ is a polyconvex, frame-indifferent elastic-energy density satisfying the growth condition 
	\[
	W_{\mathrm{el}}(A) \ge -c_1 + c_2\abs{A}^q\quad \textrm{with}\quad q>p/(p-2),
	\]
and that we are given a functional $\mathcal{F}_{\epsilon}(\{s_j\}_{j=1}^N)$ which is w.l.s.c. in $(L^p)^N$ and bounded below according to 
$\mathcal{F}_{\epsilon}(\{s_j\}_{j=1}^N)\geq\epsilon\sum_{i=1}^N\|s_j\|_p^p$, for some $\epsilon > 0$.
	
Then, writing $E_{\mathrm{el}}\left(y, \{s_j\}_{j=1}^N\right):= \int_{\Omega} W_{\mathrm{el}}\left(\nabla y\,F_\pl^{-1}\right)\dx$, where the plastic-deformation tensor is given by $F_\pl = \Id + \sum_{j=1}^N s_j \otimes m_j$, the regularised single-plane elastoplastic energy with $L^1$-hardening,
	\[
	\mathcal{E}^1_{\mathrm{reg}}\left(y, \{s_j\}_{j=1}^N\right) =  \left\{\begin{array}{ccl}
	\mathcal{E}^1_{\mathrm{rel}}\left(y, \{s_j\}_{j=1}^N\right) + \mathcal{F}_\eps & : & \abs{s_j}\abs{s_i} = 0~\textrm{a.e., for}~i\neq j\\
	& & \\
	+\infty & : & \mathrm{otherwise}
	\end{array} \right.,
	\]
	admits a minimiser $\left(y^{\ast},\{s^{\ast}_j\}_{j=1}^N\right)$ in the class $(W^{1,r}(\Omega))^3\times (L^p(\Omega))^N$, where $\frac{1}{r}=\frac{1}{p}+\frac{1}{q}$.
\end{theorem}

\begin{remark}
Proving Theorem \ref{thm:main_alt} is equivalent to proving Theorem 1, since $\mathcal{G}_{\mathrm{lam}}(\{s_j\})$ is semi-norm equivalent to $\mathcal{G}(\{s_j\})$, and $ |F_{\pl}|_{\textrm{lam}}$ is norm-equivalent to $|F_{\pl}|$, for fixed $b_{1j}, b_{2j}$, as we showed in (\cite{ang2}, Prop. 2.1).
\end{remark}

\begin{remark}
Geometrically linear elasticity also works here -- simply replace $W^{1,r}$ with $W^{1,p}$ in the statement of the result. 
\end{remark}

\subsection{$L^p$-hardening, $p>1$}

We now derive upper and lower bounds for the l.s.c. envelope of the single-slip energy, $\mathcal{E}^p$, such that the upper bound is obtained by approximating a given single-plane test funtion with single-slip laminates. For the lamination procedure we needn't pay attention to the elastic energy, since it is appropriately continuous with respect to the single-slip approximation which we use -- see the proof of Theorem 5.1 in \cite{ang}.
\subsubsection{Lower bound}
First of all, the functional 
\be
E_{\mathrm{lb}}(y,\{s_j\}) :=  \left\{
\begin{array}{ccl}
	\int_{\Omega} W_{\textrm{el}}(F^{\mathrm{el}})\dx + \sigma\int_{\Omega}\mathcal{G}_{\mathrm{lam}}(\{s_j\}) + \tau\sum_{j=1}^N\int_{\Omega} (|c_{1j}| + |c_{2j}|)^p\dx & : & \textbf{(RSC)}~\textrm{holds},\\
	& & \\
	+\infty & : & \textrm{otherwise}
	\label{rel_lb}
\end{array}
\right.
\ee
is a good lower bound on the l.s.c. envelope of $E$, in the sense that 
\begin{itemize}
\item It agrees with $\mathcal{E}^p$ on the set of smooth single-slip test functions.
\item It is $W^{1,r}\times L^p$-weakly l.s.c on the set of single-plane test functions, due to the convexity of the hardening term and the definition of the laminated curl as a total variation, and also the fact that Lemma \ref{wk_minors} (below) takes care of the weak continuity of the minors of $F_{\mathrm{el}}$.
\item \textbf{(RSC)} is preserved along weakly converging bounded-energy sequences -- see Lemma \ref{lem:main}.
\end{itemize}   

\begin{remark}
Note that the two plastic contributions to this lower bound can be arrived at by separately optimising the hardening energy (pointwise), and then the dislocation energy via approximation of a single-plane test function with single-slip laminates in two different ways: in general, one cannot reach this lower bound by optimising the total plastic energy with such laminates in one fell swoop, in contrast to the $L^1$-case -- see below.
\end{remark}

\subsubsection{Upper bounds}
Our first (and coarsest) upper bound can be obtained by weakly approximating a given single-plane slip, $\beta=\sum_{i,j}c_{ij}b_{ij}\otimes m_j$, with alternating flat slices of single-slip in the $b_{1j}$ and $b_{2j}$-directions, as we did in \cite{ang} for the $p=1$ case. Here, however, the relative thickness (weight) of the alternate slices in a bi-layer is allowed to vary from one bi-layer to the next -- in the $L^1$-case the laminated plastic energy is indifferent to the weighting, due to the 1-homogeneity of the hardening and the curl, but for $p>1$ there is, for each bi-layer, a unique optimal choice of the weighting, as we now show.

Thus, using notation from Theorem~\ref{thm:nonlinear}, we laminate a smooth, single-plane $\beta$ by filling each $\Omega_j$ with a stack of bi-layers, each parallel to $m_j^{\perp}$, and having thickness $\frac{1}{2^n}$, $n\in\mathbb{N}$, and then defining on each successive bi-layer an alternating (as we move in the $m_j$-direction), single-slip $\beta_n^j$, by
\be
\beta_n^j =
 \left\{
\begin{array}{rcl}
\lambda_j^{-1} c_{1j} b_{1j}\otimes m_j & : & \textrm{top slice} \left(\textrm{of thickness}~\frac{\lambda_j}{2^n}\right)\\
 & & \\
(1-\lambda_j)^{-1} c_{2j} b_{2j}\otimes m_j & : & \textrm{bottom slice} \left(\textrm{of thickness}~\frac{(1-\lambda_j)}{2^n}\right)\\
\end{array}
\right.,\label{slice}
\ee  
where the $c_{ij}$ are evaluated on the dividing-plane of the bi-layer in (\ref{slice}), and, for definiteness, the $(n+1)$-th laminate is obtained from the $n$-th by bisecting each of the bi-layers along a slip plane. The scaling with $\lambda_j(t_j)\in(0,1)$ (resp. ($1-\lambda_j(t_j)$)) guarantess that $\beta_n^j\rightharpoonup\beta^j$ in $L^p(\Omega_j)$ as $n\ra\infty$. Here, $t_j$ represents a coordinate running in the $m_j$-direction, and the weighting function $\lambda_j(\cdot)$ is assumed continuous. 

With this definition, we get, with abuse of notation
\be\label{E_lam}
\mathcal{E}^p(y_n,\beta_n)\ra \int_{\Omega} W_{\textrm{el}}(F^{\mathrm{el}}(y,\beta))\dx + \tau\sum_{j=1}^N\left(\int_{\Omega_j}\frac{|c_{1j}|^p}{\lambda_j^{p-1}} + \frac{|c_{2j}|^p}{(1-\lambda_j)^{p-1}}\right)\dx + \sigma\int_{\Omega}\mathcal{G}_{\mathrm{lam}}(F_\pl(\beta)),
\ee
as  $n\ra\infty$, for an appropriate zig-zag perturbation, $y_n$, of $y$ which accommodates the lamination -- see Theorem 5.1 of \cite{ang} for details, and for the convergence of the elastic energy under such perturbations. Note also that the limiting curl of $\beta_n$ is independent of $\lambda(\cdot)$ (where $\lambda|_{\Omega_j} = \lambda_j$), by homogeneity, and that this limit is just the laminated curl that appears in the $p=1$ problem.

Next, elementary calculus allows us to optimise each $\lambda_j(t_j)$ in (\ref{E_lam}), which results in 
\[
\lambda_j^{\textrm{opt}}(t_j) = \frac{\|c_{1j}\|_{L^p\left(\Omega_j^{t_j}\right)}}{\|c_{1j}\|_{L^p\left(\Omega_j^{t_j}\right)}  + \|c_{2j}\|_{L^p\left(\Omega_j^{t_j}\right)}},
\]
where $\Omega_j^{t_j}$ is a 2-d slice through $\Omega_j$ at the level $t_j$.

Thus, inserting these optimal $\lambda_j$ into the right-hand side of (\ref{E_lam}) shows that 
\be
E_{\textrm{ub}}^{(1)}(y,\{s_j\}):= 
\left\{
\begin{array}{l}
\int_{\Omega} W_{\textrm{el}}(F^{\mathrm{el}}(y,\{s_j\})\dx + \tau\sum_{j=1}^N\int\left(\|c_{1j}\|_{L^p\left(\Omega_j^{t_j}\right)}  + \|c_{2j}\|_{L^p\left(\Omega_j^{t_j}\right)}\right)^p~dt_j\\
 \\
+ \sigma\int_{\Omega}\mathcal{G}_{\mathrm{lam}}(\{s_j\})	~~~\qquad: \qquad\textbf{(RSC)}~\textrm{holds},\\
\\
+\infty\qquad\qquad\qquad\qquad: \qquad\textrm{otherwise}
\label{ub1}
\end{array}
\right.
\ee
is an upper bound for the sought-after relaxation.

\begin{remark}
This upper bound also holds for test functions which satisfy just the mild slip-patch-regularity condition of Theorem 5.1 in \cite{ang}, since we showed there that one can mollify a single-plane $\beta$ strongly continuously in $L^p$, such that the laminated curl essentially does not increase. 
\end{remark}

\begin{remark}
If $c_{1j}\propto c_{2j}$ globally on $\Omega_j$, for each $j=1,\ldots N$, then $E_{\textrm{lb}}(y,\{s_j\}) = E_{\textrm{ub}}^{(1)}(y,\{s_j\})$, by Cauchy-Schwarz, so we know precisely what the relaxation looks like for such nice test functions (provided they are reasonably smooth).
\end{remark}

\begin{remark}
In uniform-shear experiments of the type described in \cite{ang}, one expects minimising slips to point in the direction of the overall shear everywhere (and uniqueness of minimisers for a strictly convex hardening penalty, along with appropriate reflection symmetry of the energy, proves this, at least for geometrically linear elasticity). Thus, in such cases, it looks as though the proportionality condition just mentioned can be taken to hold w.l.o.g. when looking for minimisers, and hence one might conjecture that $E_{\mathrm{lb}}$ is the correct relaxed energy.
\end{remark}

We can obtain a better upper bound by allowing the bi-layer weighting $\lambda$, now assumed to be a Lipschitz function of $x\in\Omega$, to be non-constant in activated slip-planes. Using the co-area formula to take care of the curl generated across the (in general) undulating dividing-surfaces of the bi-layers, we obtain the following laminated energy in the limit $n\ra\infty$:
\begin{eqnarray}\label{wobbly_lam}
I(y,\{s_j\},\lambda) & := & \int_{\Omega} W_{\textrm{el}}(F^{\mathrm{el}})\dx + \sum_{j=1}^N\left\{\tau\int_{\Omega_j}\frac{|c_{1j}|^p}{\lambda_j^{p-1}} + \frac{|c_{2j}|^p}{(1-\lambda_j)^{p-1}}\dx\right.\nonumber\\
 & & + \sigma\int_{\Omega_j}\left|\nabla_{m_j^{\perp}} c_{1j} - c_{1j}\nabla_{m_j^{\perp}}\ln\lambda_j\right| +  \left|c_{1j}\nabla_{m_j^{\perp}}\ln\lambda_j\right|\dx\nonumber\\
 & & + \left.\sigma\int_{\Omega_j}\left|\nabla_{m_j^{\perp}} c_{2j} - c_{2j}\nabla_{m_j^{\perp}}\ln(1-\lambda_j)\right| +  \left|c_{2j}\nabla_{m_j^{\perp}}\ln(1-\lambda_j)\right|\dx\right\}.
\end{eqnarray}

Here, assuming $\lambda(x)$ to be Lipschitz continuous and bounded away from $0$ and $1$, along with $\beta\in C^1$, allows one to handle the displacement-perturbation as in \cite{ang}: that is to say, the elastic energy is still continuous with respect to lamination.

Clearly, by the triangle inequality, the curl-type contribution to (\ref{wobbly_lam})
is no smaller than the laminated curl, $\int_{\Omega}\mathcal{G}_{\mathrm{lam}}(\{s_j\})$, and an easy calculation shows that the laminated curl is reached iff at each point of $\Omega_j$ either $\nabla_{m_j^{\perp}}\lambda_j$ vanishes or $\nabla_{m_j^{\perp}} c_{1j}$ and $\nabla_{m_j^{\perp}} c_{2j}$ point in opposite directions, the latter of which will not hold for general test functions, of course. Choosing $\lambda_j$ constant on slip planes minimises the curl contribution, but the corresponding hardening contribution will then, in general, be some way from the pointwise-optimal hardening energy which appears in the lower bound (\ref{rel_lb}). Thus, for $p>1$, there is a trade-off between the hardening and dislocation energy of an undulating laminate, and one would have to be supremely optimistic to expect an explicit formula for the optimal energy.

One might conjecture that the upper-bound
\be
E_{\textrm{ub}}^{(2)}(y,\{s_j\}) := \inf\left\{I(y,\{s_j\},\lambda):\lambda\in\mathrm{Lip}(\Omega),~0<\lambda(x)<1\right\}
\ee
is in fact the relaxation we're looking for, at least for sufficiently smooth $s_j$. We have not, however, been able to prove the convexity of this uncountable {\em infimum}.

By construction, if $\mathcal{E}^p_{\mathrm{rel}}(y,F_\pl(\beta))$ denotes the single-slip-to-single-plane relaxation of $E$ (i.e., $\mathcal{E}^p_{\mathrm{rel}}(x) = \inf\{\liminf_{x_n\rightharpoonup x}E(x_n)\}$, with the {\em infimum} taken over all single-plane sequences weakly converging to a given single-plane $x$), then we have 
\be\label{energy_chain}
E_{\mathrm{lb}}(y,\{s_j\})\leq \mathcal{E}^p_{\mathrm{rel}}(y,\{s_j\})\leq E_{\textrm{ub}}^{(2)}(y,\{s_j\})\leq E_{\textrm{ub}}^{(1)}(y,\{s_j\}),
\ee
on single-plane test functions, so that, in particular, $\mathcal{E}^p_{\mathrm{rel}}$, extended to $+\infty$ when $\textbf{(RSC)}$ is violated, satisfies the conditions of Theorem \ref{thm:main}. Moreover, the  last inequality is in general strict, by the remarks above and the following example.

\begin{example}\label{ex:1}
Here is an essentially 1-d example of a single-plane $\beta$ for which the optimal flat lamination can be bettered by a sigmoidal one for $p=2$. 

Let $\Omega = (0,1)\times(-X,X)\times(0,1)\in\mathbb{R}^3$, with active slip-normal pointing in the $x_3$-direction everywhere, and for the material parameters let $\sigma=\tau=1$. Denote the single-slip strain by $\beta = c_1b_1 + c_2b_2$, such that
\[
c_1(x) =
 \left\{
\begin{array}{lcl}
 1 & : & x_2<0  \\
  & &    \\
 \eps & : & x_2\geq0 \\
\end{array}
\right.,\qquad
c_2(x) =
 \left\{
\begin{array}{lcl}
 \eps & : & x_2<0  \\
  & &    \\
 1 & : & x_2\geq0 \\
\end{array}
\right.,
\]
for positive constants $\eps\ll 1$ and $X\gg 1$.

Now, by symmetry, the optimal flat lamination is obtained by taking $\lambda=\frac{1}{2}$
 everywhere, and the plastic part of $E_{\textrm{ub}}^{(1)}(u,F_\pl(\beta))$ is readily calculated to be $2(2X(1+\eps)^2 + (1-\eps))$.
 
By choosing, instead, an appropriate non-constant $\lambda$, and $X$ large enough, we can almost halve the plastic energy, coming close to the plastic part of $E_{\textrm{lb}}(u,\{s_j\})$ (note that this is a case where $\nabla_{12}c_1$ and $\nabla_{12}c_2$ point in opposite directions everywhere). 

Specifically, we choose our Lipschitz $\lambda$ to be
\[
\lambda(x) =
 \left\{
\begin{array}{lcl}
 \frac{1}{1+\eps} & : & x_2\leq -1\\
   & & \\
 \textrm{affine} & : & -1< x_2 < 1\\
   & & \\ 
 \frac{\eps}{1+\eps} & : & x_2\geq 1 \\
\end{array}
\right..
\]

Now, the $L^2$-part of the plastic energy in $I(u,\{s_j\},\lambda)$ is calculated to be
$2\{(X-1)(1+\eps)^2 + H_1(\eps)\}$, for some $H_1(\eps)\in[1+\eps+2\eps^2, 2+\eps(1+\eps)]$, while the curl part is of the form $H_2(\epsilon)$ (i.e. independent of $X$). Thus, the total plastic energy of the jagged laminate may be written as $2(X-1)(1+\eps)^2 + 2H_1(\eps) + H_2(\eps)$. For $X$ large and $\epsilon$ small, this is roughly one-half of the plastic energy of the optimal flat laminate, as claimed.
\end{example}

\section{Existence of minimisers} \label{sec:existence}
We will use the direct method of the calculus of variations to prove the existence of minimisers for any single-plane energy which satisfies the requirements stated in Theorem \ref{thm:main}.

\subsection{Minimising sequences}
\begin{proposition}\label{hoelder}
Assume that the conditions of Theorem \ref{thm:main} are satisfied, and consider a finite-energy minimising sequence $(y_k, \{s_j\}_k)$. We then have, up to taking a subsequence, $\{s_j\}_k \weak s_j$ in $L^p$ and $y_k \weak y$ in $W^{1,r}$, for some $s_j$ and $y$, where $\frac{1}{r} =  \frac{1}{p} + \frac{1}{q} >0$.
\end{proposition}
\begin{proof}
The first part of the claim is immediate by inspection, while the second follows from 
\[
\norm{\nabla y F_\pl^{-1}}_{L^q} \ge \frac{\norm{\nabla y}_{L^r}}{\norm{F_\pl}_{L^p}},
\]
which in turn follows from H\"older's inequality, since $F_\pl^{-1} = \Id - \sum_{j=1}^N s_j\otimes m_j$ almost everywhere along the minimising sequence, by the single-plane condition -- see~\cite{MM}.
\end{proof}

\subsection{Weak convergence of minors}

\begin{lemma}\label{wk_minors}
If $p>2$ and $q>\frac{p}{p-2}$, then for a minimising sequence as in Proposition \ref{hoelder}, we have, up to a subsequence,
$$
\Mino_{1,2,3}((F_{\mathrm{el}})_k) := \Mino_{1,2,3}(\nabla y_k (F_\pl^{-1})_k) \weak \Mino_{1,2,3}(\nabla y F_\pl^{-1})
$$
in $L^1(\Omega)$, i.e., each sequence of minors converges weakly to the minor of the limiting elastic deformation along the minimising sequence.
\end{lemma}
\begin{proof} In what follows, we will often drop sequence subscripts for notational convenience, and $r$ will be determined by the model parameters $p$ and $q$ as in Proposition \ref{hoelder}.

Along our energy-bounded minimising sequence, the plastic-strain tensor has the form
\begin{align*}
F_\pl &= \Id + \sum_{j=1}^N s^j(x) \otimes m^j \\
&= \begin{pmatrix}
1 +  \sum_{j=1}^N s^j_1(x)\, m^j_1   & \sum_{j=1}^N s^j_1(x)\, m^j_2  & \sum_{j=1}^N s^j_1(x)\, m^j_3 \\
  \sum_{j=1}^N s^j_2(x)\, m^j_1   & 1 + \sum_{j=1}^N s^j_2(x)\, m^j_2  & \sum_{j=1}^N s^j_2(x)\, m^j_3 \\
  \sum_{j=1}^N s^j_3(x)\, m^j_1   & \sum_{j=1}^N s^j_3(x)\, m^j_2  & 1 +\sum_{j=1}^N s^j_3(x)\, m^j_3
\end{pmatrix},
\end{align*}
with $s^j \perp m^j$ and $\abs{s^i}\abs{s^j} = 0$ for $i\neq j$, a.e. in $\Omega$. 

Moreover, due to the side condition and  $s^j \perp m^j$, we have 
\[
F_\pl^{-1} = \Id - \sum_{j=1}^N s^j(x) \otimes m^j,\quad\textrm{a.e. in}~\Omega.
\]
In order to prove convergence of minors, we need to control the integrability of the entries in the elastic-strain tensor, which reads
\begin{align*}
F_{\textrm{el}} &= \nabla yF_\pl^{-1} \\
&= \begin{pmatrix}
y_{1,1} - \sum_{j=1}^Nm_1^j\sum_{k=1}^3 s_k^j y_{1,k} & y_{1,2} -\sum_{j=1}^Nm_2^j\sum_{k=1}^3 s_k^j y_{1,k} & y_{1,3} -\sum_{j=1}^{N}m_3^j\sum_{k=1}^3 s_k^j y_{1,k} \\
y_{2,1} - \sum_{j=1}^{N}m_1^j\sum_{k=1}^3 s_k^j y_{2,k} & y_{2,2} -\sum_{j=1}^{N}m_2^j\sum_{k=1}^3 s_k^j y_{2,k} & y_{2,3} -\sum_{j=1}^{N}m_3^j\sum_{k=1}^3 s_k^j y_{2,k} \\
y_{3,1} - \sum_{j=1}^{N} m_1^j\sum_{k=1}^3 s_k^j y_{3,k} & y_{3,2} -\sum_{j=1}^{N}m_2^j\sum_{k=1}^3 s_k^j y_{3,k} & y_{3,3} -\sum_{j=1}^{N}m_3^j\sum_{k=1}^3 s_k^j y_{3,k}
\end{pmatrix},
\end{align*}
such that all of the matrix entries are in $L^q$, due to the growth condition on the elastic energy density, $W_{\textrm{el}}$. 
\\~\\
{\em The $1\times 1$-minors.} The `bad' terms appearing in the $1\times 1$-minors are objects of the form $M_{1}:=\nabla y_i\cdot s^j$, such that, by the argument of Proposition \ref{hoelder}, and another application of H\"{o}lder's inequality, $M_1\in L^z$, with $z$ given by $\frac{1}{z}=\frac{1}{r}+\frac{1}{p}=\frac{2}{p}+\frac{1}{q}<1$. $M_1$ is thus equi-integrable along the minimising sequence.

We now apply a div-curl argument to $M_1$. Clearly, $\nabla y$ is curl-free, and with the aid of an orthonormal basis $\mathcal{O}^j=(m^j,e_1^j, e_2^j)$ adapted to the $j$-th slip normal, we have
\bq
|\textrm{div} s^j|  & = & |\p_{m^j}\tilde{s}_m^j + 	
\p_{e_1^j}\tilde{s}_1^j + \p_{e_2^j}\tilde{s}_2^j|\nonumber\\
    & = & |\p_{e_1^j}\tilde{s}_1^j + \p_{e_2^j}\tilde{s}_2^j|\nonumber\\
    & \leq & 2|\nabla_{m_{\perp}^j}s^j|\\
    & \leq & C,
\eq
where $\tilde{s}_i^j$ are the components of $s^j$ w.r.t. $\mathcal{O}^j$, and $C$ is a constant.

Thus, passing to a subsequence, we get $\div s^j_k\ra \div s^j\in W^{-1,1}$ for some $s^j$ as $k\rightarrow\infty$, since weak-$*$ convergence of Radon measures implies strong convergence in $W^{-1,1}(\Omega)=(W_0^{1,\infty})^{\ast}(\Omega)$, and we may thus appeal to the div-curl Theorem of \cite{cdm} to get the required convergence of the $1\times 1$-minors.
\\~\\
{\em The $2\times 2$-minors.} For the $2\times 2$-minors, first notice that
\bq
\textrm{cof}~(\nabla y\cdot F^{-1}_\pl) & = &\textrm{cof}~\nabla y\cdot\textrm{cof} ~F^{-1}_\pl\\
 & = & (\textrm{cof}~\nabla y)\cdot F^T_\pl,
\eq
a.e., since $\det F_\pl^{-1}=1$, a.e..

Thus, since $\textrm{cof}~(\nabla y\cdot F^{-1}_\pl)\in L^q$, by our assumptions on $W_{\mathrm{pl}}$, we have, again by \cite{MM}, $\textrm{cof}~\nabla y\in L^r$. In other words, the $2\times 2$-minors of $\nabla y$ are also in $L^r$.

Now, according to Lemma 2.4 of \cite{MM}, the $2\times 2$-minors of $\nabla y\cdot F^{-1}_\pl$ can be written as $\frac{\det H}{\det F_\pl} ~~(=\det H)$, for a $3\times 3$ matrix $H$ which consists of two rows of $\nabla y$, and one of $F_\pl$: for example, something of the form
\[
H=
\begin{pmatrix}
y_{1,1} & y_{1,2} & y_{1,3}\\
y_{2,1} & y_{2,2} & y_{2,3}\\
\sum_{j=1}^N s^j_3(x)\, m^j_1   & \sum_{j=1}^N s^j_3(x)\, m^j_2  & 1 +\sum_{j=1}^N s^j_3(x)\, m^j_3
\end{pmatrix}.
\]
Expanding the determinant of this example about the 3rd row, we see by the standard div-curl lemma that the $2\times2$-minors in the expansion can be taken to converge weakly in $L^1$ to the correct limit. Since they are bounded in $L^r$, by the above, we also have these minors converging weakly in $L^r$ to the correct limit, along a subsequence.

Next note that $\det H = \eta\cdot\xi$, where $\eta$ is a vector of $2\times2$-minors of $\nabla y$, and 
\[\xi = \left(\sum_{j=1}^N s^j_3(x)\, m^j_1,~\sum_{j=1}^N s^j_3(x)\, m^j_2,~ 1 +\sum_{j=1}^N s^j_3(x)\, m^j_3\right)^T. 
\]
Thus, $\div\eta = 0$, while $\curl(s_3^j(m_1^j,m_2^j,m_3^j)^T)$ takes the form $(0,\p_{e_2^j}s_3^j ,- \p_{e_1^j}s_3^j)$ in a frame $\mathcal{O}^j$ as above, which is once again a measure dominated by the dislocation energy. Taking $q>\frac{p}{p-2}$ ensures that $\xi$ and $\eta$ converge in conjugate Lebesgue spaces, and hence all the conditions of the div-curl Theorem in \cite{cdm} are satisfied, giving the required convergence along a subsequence.
\\~\\
{\em The $3\times 3$-minor.} This is just $\det\nabla y$. We can expand this $3\times3$-determinant as the dot product of a divergence-free vector of $2\times2$-minors and a curl-free vector, both of which are bounded in $L^r$, as above. By our stated assumption on $p$ and $q$, this ensures that the two vectors converge weakly in conjugate Lebesque spaces, and so once more we may apply \cite{cdm} to get the required convergence along a suitable subsequence.
\end{proof}

\subsection{Preservation of the single-plane condition under weak-$L^p$ convergence}

Consider now a sequence $\{s_j^k\}$, $j=1,\ldots,N$, $k\in\mathbb{N}$, of relaxed slips with the following properties: 
\begin{enumerate}[(P1)]
\item \label{enum:bjsecond} For each $k$, $\{s_j^k\}_{j=1}^N$ satisfies our \emph{single-plane condition}, i.e.,
\[s_j^k(x)\in m_j^{\perp}\quad\textrm{and}\quad |s_i^k(x)||s_j^k(x)|=0\quad\textrm{for a.e. }~x\in\Omega,~i\neq j.
\]	
\item\label{enum:bjfirst} $s_j^k \weak s_j$ in $L^p$, $p\ge 1$, for some $s_j$, as $k\rightarrow\infty$.
\item \label{enum:bjlast} The density of geometrically necessary dislocations is uniformly summable, i.e.,
\begin{align*}
\mathcal{G}(\{s_j^k\}_{j=1}^N)\leq K,
\end{align*}
for all $k\in\mathbb{N}$ and some $K<\infty$, where $\mathcal{G}(\cdot)$ is the modified curl defined in (\ref{mod_curl}). 
\end{enumerate}

Our closedness result is then the following.

\begin{lemma} \label{lem:main}
Assume we have a sequence $\{s_j\}^k, k\in\mathbb{N}$, satisfying properties~(P\ref{enum:bjsecond})-(P\ref{enum:bjlast}) above. Then the limit $\{s_j\}$ satisfies the single-plane condition, (P\ref{enum:bjsecond}). 
\end{lemma}
\begin{proof}
First of all, we clearly have $s_j(x)\perp m_j$, for a.e  $x\in\Omega$, by weak convergence. For a contradiction, assume that there exists a measurable set $S \subset \Omega$ with $\abs{S}>0$, such that on $S$ at least two of the $s_j$ are non-vanishing. Without loss of generality, we can take these to be $s_1$ and $s_2$. By approximation of measurable sets by closed sets from the inside, there exists a set $S'\subset S$ and a $\delta>0$, such that on $S'$ we have $\abs{s_1}, \abs{s_2} \ge \delta$ and $\abs{S'}\ge \delta$. 

Now introduce coordinates $(x_1,x_2,x_3)$ which are adapted to the (not necessarily orthonormal) frame $(m_1, m_2, m_1\times m_2)$. By approximation of measurable sets by open sets from the outside, for a given $\eps > 0$, we can find a finite collection of (open) parallelepipeds, aligned with the coordinate mesh, the union of which we denote by $V$, such that 
\begin{equation} \label{eq:outer_est}
\int_{V \setminus S'} \abs{s_1}+\abs{s_2}\dx \le \eps\quad \textrm{and} \quad S' \subset V.
\end{equation}
Note that we can always assume $V \subset\subset \Omega$, by subtracting a thin collar-neighbourhood of the Lipschitz boundary, $\p\Omega$, if necessary,.

We now fill $V$ with finitely many non-overlapping parallelepipeds $\{C_i\}_{i=1}^L$ of edge-length no larger than an arbitrary $l>0$, once more aligned with the $x_i$-coordinate mesh. By the assumed weak convergence of $s_j^k$, we have
\begin{equation} \label{eq:moll}
\liminf_{k\to\infty} \int_{C_i\cap S'}  \abs{s_j^k}\dx \ge \int_{C_i\cap S'} \abs{s_j}\dx, \quad j=1,2.
\end{equation}
 
Moreover, by Lemma~\ref{lemma:curl} (below), there exists a geometric constant $c>0$ such that 
$$
\mathcal{G}(\{s_j^k\})\ge  c\int_V \abs{\frac{\partial \abs{s_2^k}}{\partial x_1}}  +  \abs{\frac{\partial \abs{s_1^k}}{\partial x_2}}\dx, 
$$
and thus, using~Lemma~\ref{lem:int_min} (even further below), 
$$
\mathcal{G}(\{s_j^k\}) \ge \frac{c}{l}\left(\sum_i \min \left\{\int_{C_i\cap S'}\abs{s_2^k}\dx, \int_{C_i\cap S'}\abs{s_1^k}\dx\right\}\right).
$$
Now, keeping $l$ fixed, pick an arbitrary $\hat{\eps}>0$, and choose $M(\hat{\eps},l)$ large enough such that, for all $i$ and for $j=1,2$, we have
$$
\int_{C_i\cap S'} \abs{s_j^k}\dx\ge \int_{C_i\cap S'} \abs{s_j}\dx - \frac{\hat{\eps}l}{Lc}
$$
for any $k\ge M$. We thus obtain for any $k$ large enough that
\begin{align*}
\mathcal{G}(\{s_j^k\}) &\ge \frac{c}{l}\left(\sum_i \min \left\{\int_{C_i\cap S'}\abs{s_2}\dx, \int_{C_i\cap S'}\abs{s_1}\dx\right\}\right) - \hat{\eps} \\
&\ge \frac{c}{l}\delta\abs{S'} - \hat{\eps},
\end{align*}
by~\eqref{eq:outer_est}.

Finally, by taking $\hat{\eps}$ and $l$ small enough, keeping $\delta>0$ and $S'$ fixed, we can thus make the $\curl$ arbitrarily large, which is a contradiction to energy boundedness.
\end{proof}

\begin{lemma} \label{lemma:curl}
Consider an open set $V \subset\subset \Omega$. We then have, for any $\{s_j\}$ on $\Omega$,
\begin{equation}\label{eq:smooth_curl}
\mathcal{G}(\{s_j\})|_V \ge c  \int_V \abs{\frac{\partial \abs{s_2}}{\partial x_1}}  +  \abs{\frac{\partial \abs{s_1}}{\partial x_2}} + \abs{\frac{\partial \abs{s_2}}{\partial x_3}}  +  \abs{\frac{\partial \abs{s_1}}{\partial x_3}}\dx,
\end{equation}
for some geometric constant $c>0$, whereby the coordinates $x_1,x_2$ and $x_3$ are adapted to $m_1, m_2$ and $m_1\times m_2$, respectively. 
\end{lemma}
\begin{proof}
First assume the $s_i$ are smooth. Next, it is convenient to calculate the row-wise curl of $\beta$ (as above) by first transforming from the $(m_1,m_2,m_1\times m_2)$-frame to an orthonormal one given by first shearing $m_2$ in the $(m_1,m_2)$-plane such that it becomes orthogonal to $m_1$, and then shearing $m_3$ to make it orthogonal to the $(m_1,m_2)$-plane. The corresponding transformation matrix is thus of the form
\begin{equation}
M = \left( 
\begin{array}{ccc}
1 & 0 & \sin\phi\cos\psi\\
0 & 1 & \sin\phi\sin\psi\\
0 & 0 & \cos\phi
\end{array}
\right)
\left(
\begin{array}{ccc}
1 & \cos\theta & 0\\
0 & \sin\theta & 0\\
0 & 0 & 1
\end{array}
\right),
\end{equation}
for some angles $\theta, \phi$ and $\psi$.

Applying this to $s_1\otimes m_1$, and then taking the 3-component of each of the row-curls, gives a pointwise estimate which leads to
\begin{equation}\label{eq:curl_s1}
\int_V \abs{\nabla_{m_1^{\perp}} s_1}\geq c_1\int_V \abs{\partial_{x_2}s_1^3} + \abs{\partial_{x_2}\left(s_1^2 + c_2s_1^3\right)} + \abs{\partial_{x_2}\left(s_1^1 + c_3s_1^2 + c_4s_1^3\right)}\dx,
\end{equation}
for some $c_i(\theta,\phi,\psi)>0$, where the $s_1^i$ are components of $s_1$ in the original $(m_1, m_2,m_1\times m_2)$-frame.
An entirely analogous inequality can be obtained for $\curl(s_2\otimes m_2)$ by making the subscript switch $1\leftrightarrow2$ in (\ref{eq:curl_s1}). By taking the 2-components of the row-curls of $s_1\otimes m_1$ and $s_2\otimes m_2$ we also get analogous inequalities with $x_3$- rather than $x_2$-derivatives everywhere on the rhs (in both cases). Repeated application of the triangle inequality now gives (\ref{eq:smooth_curl}).

Finally, we remove the assumption of smoothness on $\{s_j\}$. Thus, by the first part of the proof of Proposition 3.3 in \cite{ang2} (eq. 3.23, in particular, which doesn't depend on slip-patch-boundary regularity), we can mollify the $s_i$ to get
smooth $(s_i)_{\eps}$ such that
\begin{equation} \label{eq:approx}
\|s_j - (s_j)_{\eps}\|_{L^p(V)}\leq\eps \quad \mathrm{and} \quad \int_V\mathcal{G}(s_j)|\geq \int_V\mathcal{G}(s_j)_{\eps}| -\eps.
\end{equation}
Then we apply (\ref{eq:smooth_curl}) to $(s_j)_{\eps}$ and appeal to the $L^p$-lower semicontinuity of all the derivative terms on the right-hand side as $\eps\rightarrow 0$ to get the desired result.
\end{proof}

\begin{remark}\label{rem:comp}
One also obtains inequalities analogous to (\ref{eq:smooth_curl}) by replacing the absolute values of $s_1$ and $s_2$ with any of their components, which will be useful below.
\end{remark}

\begin{lemma}\label{lem:int_min}
Suppose we have a single-plane sequence $\{s_j\}^k$ with uniformly bounded dislocation-and-hardening energy, with active slips $s_1^j$, $s_2^j$ which mix in the limit on a covering of parallelepipeds $V = \cup_i C_i\supset S'$, as in Theorem \ref{lem:main}. Then, for sufficiently large $j$, there exists $c>0$ such that
\begin{equation}\label{eq:L1_est}
\int_V \abs{\frac{\partial \abs{s_{2}^j}}{\partial x_1}}  +  \abs{\frac{\partial \abs{s_{1}^j}}{\partial x_2}}\dx\ge\frac{c}{l}\sum_i \min \left\{\int_{C_i\cap S'}\abs{s_2^j}\dx, \int_{C_i\cap S'}\abs{s_1^j}\dx\right\},
\end{equation}
for coordinates $x_i$ aligned with the $(m_1,m_2,m_1\times m_2)$-frame, provided the $x_1$- and $x_2$-coordinate extents of the $C_i$ are no greater than $l$, where $c$ depends on the energy bound.
\end{lemma}
\begin{proof}
Consider a parallelepiped $P = Q\times [0,t]$, with $Q$ a parallelogram aligned with the $x_1, x_2$-coordinate mesh, having coordinate extents $l_1$ and $l_2$ which are dominated by a constant $l$. To be specific, and without loss of generality, let $Q = \{(x_1,x_2) : x_1\in(0,l_1), x_2\in(0,l_2)\}$. Then the single-plane condition allows us to employ the argument of Lemma 4.3 in \cite{co}, which also works for non-orthogonal coordinates, on slices $Q\times\{\tau\}$. Thus, for each $k$ and $\tau$, we define the following subsets of $Q$:
\bq
\omega^1_{k,\tau} & = & \left\{x_1\in (0,l_1): s_1^k(x_1,x_2,\tau)\neq 0~\textrm{for a.e.}~x_2\in(0,l_2)\right\}, \\
\omega^2_{k,\tau} & = & \left\{x_2\in (0,l_2): s_2^k(x_1,x_2,\tau)\neq 0~\textrm{for a.e.}~x_1\in(0,l_1)\right\}, 
\eq
and, from the single-plane condition and basic measure theory, we conclude that for each k and a.e. $\tau\in [0,t]$ at least one of $\omega^1_{k,\tau}$ and $\omega^2_{k,\tau}$ must be a null set.
 
For the sake of argument, assume that, for a given $k$ and $\tau$, $\omega^1_{k,\tau}$ is a null set. Then, for a.e. $x_1\in(0,l_1)$, we have
\begin{equation}
\abs{s_1^k}(x_1,x_2,\tau)\leq\int_0^{l_2}\frac{\p\abs{s_1^k}}{\p x_2}(x_1,x_2',\tau)~\dx_2'.
\end{equation}
Hence, integrating w.r.t. $x_2$ and then $x_1$, we obtain
\begin{equation}
\iint_Q\abs{s_1^k}\dx\leq l_2\iint_Q\frac{\p\abs{s_1^k}}{\p x_2}\dx,
\end{equation}
and, by a similar argument for the other possible case, we therefore see that, for a.e. $\tau\in (0,t)$,
\begin{equation} \label{eq:co2}
\textrm{either}\quad\iint_{Q} \abs{\frac{\partial \abs{s_2^k}}{\partial x_1}}(\cdot,\tau)\dx \geq\frac{1}{l} \iint_{Q}\abs{s_2^k}(\cdot,\tau)
\dx,
\end{equation}
\begin{equation}
\quad\textrm{or}\quad
\iint_{Q} \abs{\frac{\partial \abs{s_1^k}}{\partial x_2}}(\cdot,\tau) \dx\geq\frac{1}{l}\iint_{Q}\abs{s_1^k}(\cdot,\tau)\dx.
\end{equation}

Define $C_i^t = C_i\cap S'\cap \{x_3=t\}$, so that (\ref{eq:co2}) implies
\begin{equation}\label{eq:co4}
\int_V \abs{\frac{\partial \abs{s_{2}^j}}{\partial x_1}}  +  \abs{\frac{\partial \abs{s_{1}^j}}{\partial x_2}}\dx\ge
\frac{1}{l}\int \sum_i\min\left\{\iint_{C_i^t} |s_1^k|\dx,\iint_{C_i^t} |s_2^k|\dx\right\}~\dt.
\end{equation}

Write $a_i^k(t) = \iint_{C_i^t} \abs{s_1^k}\dx$ and $b_i^k(t) = \iint_{C_i^t} \abs{s_2^k}\dx$. Then, in order to apply (\ref{eq:co4}), it suffices to show that
\begin{equation} \label{eq:co3}
\int\sum_i\min\left\{\abs{a_i^k}, \abs{b_i^k}\right\}~\dt \geq c\sum_i\min\left\{\int\abs{a_i^k}~\dt, \int\abs{b_i^k}~\dt\right\},
\end{equation}
for $k$ large, and some constant $c>0$.

Now, by assumption, $s_1^k$ and $s_2^k$ weakly converge in $L^p$, with limits greater than some $\delta>0$ in absolute value on $S'$. Let $f(t) = \mathcal{L}^2(S'\cap\{x_3 = t\})$. Then $\exists T\subset\mathbb{R}$ with $\abs{T}\geq\delta_1$, such that $f(t)\geq\delta_1$ on $T$, for some $\delta_1>0$, since $\int f(t)~dt = |S'|>0$. By taking a fine mesh of many points in $T$, and noting that Lemma \ref{lemma:curl} (resp. Remark \ref{rem:comp}) gives control on the component-wise variation of $\iint s_1^k~dx_1dx_2$ and $\iint s_2^k~dx_2$ in the $t$-direction, we see by weak convergence that $\sum_i\min\left\{\abs{a_i^k}, \abs{b_i^k}\right\}>\delta\delta_1/2$ on an arbitrarily large fraction of $T$, if $k$ is large enough (we may assume w.l.o.g. that one component of each of $s_1$ and $s_2$ is larger than $\delta$ on $S'$, then weak convergence of  $\iint s_1^k~dx_1dx_2$ and $\iint s_2^k~dx_1dx_2$, together with the curl bound, prevent $a_i^k$ and $b_i^k$ from straying too close to zero). Thus, the left-hand side of (\ref{eq:co3}) is greater than some $C$ for $k$ large. Since the right-hand side of (\ref{eq:co3}) is bounded by the (hardening) energy, we see that (\ref{eq:co3}) does in fact hold, for a constant $c$ which depends on the energy bound.
\end{proof}

\begin{remark}
While Lemma \ref{lem:main} still holds in the case $p=1$, $L^1$-hardening is of course not sufficient to guarantee $s_j^k \weak s_j$ in $L^1$ along (a subsequence of) a minimising sequence for an elastoplastic energy of the form $E$. 
\end{remark}

\subsection{Proof of Theorem \ref{thm:main}} Putting together Lemmas \ref{wk_minors}-\ref{lem:int_min} finally gives our main result, Theorem \ref{thm:main}, since we only need an upper bound on the dislocation and hardening energy along bounded-energy sequences to apply these results. Moreover, Theorem \ref{thm:main_lin} also follows, since the stronger requirement $p>2$ was only needed to control the minors in the nonlinear elastic energy. 

\subsection{Counter-example for $p=1$} We cannot expect an analogue of our main existence result for $p=1$ rate-independent dissipation, even for geometrically linear elasticity, by virtue of the following essentially 2-d example, which highlights the problem of slip concentration. 

\begin{example}\label{ex:2}
Suppose $p=1$, and consider two infinitesimally thin slip-lines (i.e. slip-planes viewed side-on) crossing at right angles, and exiting a shear sample at free boundaries. Now, by any reasonable interpretation of the relaxed side condition for measures, this construction should have infinite plastic energy. However, one can approximate these crossing slip lines weakly-$\ast$ in the space of Radon measures by fattening both of the slip lines a little, while preserving the total shear, and cutting off one of the resulting mollified shear bands just outside the region of overlap, such that $(\textbf{RSC})$ is now satisfied. The dislocation energy of such an approximation is just twice the total shear on the cut-off shear band, and in this way one can therefore pass to the weak-$\ast$ limit of crossing slip lines with finite plastic energy. In other words, for $p=1$ any reasonable interpretation of $(\textbf{RSC})$ produces a single-plane energy which fails to be lower semi-continuous.
\end{example}

\section*{Acknowledgments}
MK was supported by GA\v{C}R through projects 14-15264S and 16-34894L.
This research was partly conducted when MK held the visiting Giovanni-Prodi professorship in the
Institute of Mathematics, University of W\"{u}rzburg. Its support and hospitality are gratefully acknowledged. PWD is partially supported by the German Scholars Organization / Carl-Zeiss-Stiftung via the Wissenschaftler-R{\"u}ckkehrprogramm.

\end{document}